\documentclass[12pt]{amsart}
\usepackage[dvips]{graphicx}
\usepackage{amssymb,amsmath,amsthm}

\def\NZQ{\mathbb}               
\def\NN{{\NZQ N}}

\def\frk{\mathfrak}               

\def\Phi{{\frk N}}
\def\Pc{{\mathcal P}}

\def\I{{\mathcal I}}
\def\J{{\mathcal J}}

\def\opn#1#2{\def#1{\operatorname{#2}}} 
\opn\chara{char} \opn\length{\ell} \opn\pd{pd} \opn\rk{rk}
\opn\projdim{proj\,dim} \opn\injdim{inj\,dim} \opn\rank{rank}
\opn\depth{depth} \opn\grade{grade} \opn\height{height}
\opn\size{size}
\opn\embdim{emb\,dim} \opn\codim{codim}

\opn\Tr{Tr} \opn\bigrank{big\,rank}
\opn\superheight{superheight}\opn\lcm{lcm}
\opn\trdeg{tr\,deg}
\opn\reg{reg} \opn\lreg{lreg} \opn\ini{in} \opn\lpd{lpd}
\opn\size{size}\opn{\mult}{mult}
\opn{\Cl}{Cl}
\opn\div{div} \opn\Div{Div} \opn\cl{cl} \opn\Cl{Cl}
\opn\Spec{Spec} \opn\Supp{Supp} \opn\supp{supp} \opn\Sing{Sing}
\opn\Ass{Ass} \opn\Min{Min} \opn\cl{cl}
\opn\Ann{Ann} \opn\Rad{Rad} \opn\Soc{Soc}
\opn\Syz{Syz} \opn\Im{Im} \opn\Ker{Ker} \opn\Coker{Coker}
\opn\Am{Am} \opn\Hom{Hom} \opn\Tor{Tor} \opn\Ext{Ext}
\opn\End{End} \opn\Aut{Aut} \opn\id{id} \opn\ini{in}

\opn\nat{nat}
\opn\pff{pf}
\opn\Pf{Pf} \opn\GL{GL} \opn\SL{SL} \opn\mod{mod} \opn\ord{ord}
\opn\Gin{Gin}
\opn\Hilb{Hilb}\opn\adeg{adeg}\opn\std{std}\opn\ip{infpt}
\opn\Pol{Pol}
\opn\sat{sat}
\opn\Var{Var}
\opn\Gen{Gen}
\opn\lex{lex}
\opn\div{div}

\opn\aff{aff} \opn\con{conv} \opn\relint{relint} \opn\st{st}
\opn\lk{lk} \opn\cn{cn} \opn\core{core} \opn\vol{vol}
\opn\link{link} \opn\star{star}
\opn\gr{gr}

\def\Ac{{\mathcal A}}

\def\Qc{{\mathcal Q}}

\def\pot#1#2{#1[\kern-0.28ex[#2]\kern-0.28ex]}

\opn\dirlim{\underrightarrow{\lim}}
\opn\inivlim{\underleftarrow{\lim}}
\def\Implies{\ifmmode\Longrightarrow \else
        \unskip${}\Longrightarrow{}$\ignorespaces\fi}
\def\implies{\ifmmode\Rightarrow \else
        \unskip${}\Rightarrow{}$\ignorespaces\fi}
\def\iff{\ifmmode\Longleftrightarrow \else
        \unskip${}\Longleftrightarrow{}$\ignorespaces\fi}

\let\:=\colon

\def\NZQ{\mathbb}        
\def\NN{{\NZQ N}}

\def\frk{\mathfrak}        

\def\Phi{{\frk N}}

\newtheorem{Theorem}{Theorem}[section]
\newtheorem{Lemma}[Theorem]{Lemma}
\newtheorem{Corollary}[Theorem]{Corollary}

\theoremstyle{definition}

\begin{document}
\title{nonsimple polyominoes and prime ideals}
\author {Takayuki Hibi and
Ayesha Asloob Qureshi}
\address{Takayuki Hibi, Department of Pure and Applied Mathematics,
Graduate School of Information Science and Technology,
Osaka University, Toyonaka, Osaka 560-0043, Japan}
\email{hibi@math.sci.osaka-u.ac.jp}

\address{Ayesha Asloob Qureshi, Department of Pure and Applied Mathematics, 
Graduate School of Information Science and Technology,
Osaka University, Toyonaka, Osaka 560-0043, Japan}
\email{ayesqi@gmail.com}

\thanks{The second author was supported by JSPS Postdoctoral Fellowship for Overseas Researchers FY2014.}


\maketitle
\begin{abstract}
It is known that the polyomino ideal arising from a simple polyomino 
comes from a finite bipartite graph and, in particular, it is a prime ideal.
A class of nonsimple polyominoes $\Pc$ 
for which the polyomino ideal $I_{\Pc}$ is a prime ideal and for which 
$I_{\Pc}$ cannot come from a finite simple graph will be presented.   
\end{abstract}

\section*{Introduction}

The systematic study of the binomial ideals arising from 
polyominoes originated in the work \cite{Q} by the second author.
First, we briefly recall fundamental materials and basic terminologies 
on polyominoes and their binomial ideals.  
We refer the reader to \cite{Q} for further information
on algebra and combinatorics on polyominoes.

\medskip

{\bf (0.1)}
Let $\NN$ denote the set of nonnegative integers and 
\[
\NN^{2} = \{ (i, j) \, : \, i, j \in \NN \}.
\]
Given $a=(i,j)$ and $b=(k,\ell)$ belonging to $\NN^2$, 
we write $a < b$ if $i < k$ and $j < \ell$.

When $a < b$, we define an {\em interval} $[a,b]$ 
of $\NN^{2}$ to be
\[
[a,b]=\{c\in\NN^2 \, : \, a\leq c\leq b\} \subset \NN^{2}.
\]
For an interval $[a,b]$,  
the {\em diagonal} corners of $[a,b]$ are $a$ and $b$, and
the {\em anti-diagonal} corners of $[a,b]$ are
$c = (i,\ell)$ and $d = (k,j)$. 

\medskip

{\bf (0.2)}
A {\em cell} of $\NN^{2}$ with the lower left corner $a \in \NN^{2}$ 
is the interval $C = [a,a+(1,1)]$.  
Its {\em vertices} are $a, a+(1,0), a + (0,1)$ and $a + (1,1)$.  
Its {\em edges} are  
\[
\{a,a+(1,0)\}, \{a,a+(0,1)\},  \{a+(1,0),  a+(1,1)\},
\{a+(0,1),  a+(1,1)\}.
\]
Let $V(C)$ denote the set of vertices of $C$ and $E(C)$ 
the set of edges of $C$.
 
\medskip

{\bf (0.3)} 
Let $\Pc$ be a finite collection of cells of $\NN^2$.
Then its {\em vertex set} is $V(\Pc)=\bigcup_{C \in \Pc} V(C)$
and 
its {\em edge set} is $E(\Pc)=\bigcup_{C \in \Pc} E(C)$.
Let $C$ and $D$ be cells of $\Pc$. 
We say that $C$ and $D$ are {\em connected} 
if there exists a sequence of cells 
\[
{\mathcal C} \, : \, C = C_1, \ldots, C_m =D
\]
of $\Pc$  such that $C_i \cap C_{i+1}$ is an edge of $C_i$ for $i=1, \ldots, m-1$. 
Furthermore, if $C_i \neq C_j$ for all $i \neq j$, 
then $\mathcal{C}$ is called a {\em path} connecting $C$ with $D$. 

We say that $\Pc$
is a {\em polyomino} if any two cells of $\Pc$ are connected.
A polyomino $Q$ is a {\em subpolyomino} of $\Pc$ 
if each cell belonging to $\Qc$ belongs to $\Pc$. 

\medskip

{\bf (0.4)}
Let $A$ and $B$ be cells of $\NN^2$ for which $(i,j)$ is the lower left corner
of $A$ and $(k,\ell)$ is the lower left corner of $B$.
If $i \leq k$ and $j \leq \ell$, then
the {\em cell interval} of $A$ and $B$
is the set $[A,B]$ which consists of those cells $E$ of $\NN^{2}$
whose lower left corner $(r,s)$ satisfies 
$i\leq r \leq k$ and $j \leq s \leq \ell$.


Let $\Pc$ be a finite collection of cells of $\NN^{2}$.
We call $\Pc$ {\em row convex}
if the horizontal cell interval $[A,B]$ is contained in $\Pc$ for any cells $A$ and $B$ of $\Pc$ whose lower left corners are in horizontal position. 
Similarly one can define {\em column convex}.
We call $\Pc$ {\em convex} if it is row convex and column convex.  

An edge of $\Pc$ is a {\em free} edge if it is an edge of only one cell of $\Pc$. 
The {\em boundary} $B(\Pc)$ of $\Pc$
is the union of all free edges of $\Pc$.  
A cell $C$ of $\Pc$ is a {\em border} cell if
at least one of the edges of $C$ is a free edge.  
 
\medskip

{\bf (0.5)}
Each interval $[a,b]$ of $\NN^{2}$ can be regarded 
as a polyomino in the obvious way.  This polyomino is denoted by
$\Pc_{[a,b]}$.
Let $\Pc$ be a collection of cells of $\NN^{2}$ and $[a,b] \subset \NN^{2}$ 
an interval with $\Pc \subset \Pc_{[a,b]}$.  
Following \cite{Q}, we say that a polyomino $\Pc$ is {\em simple}
if, for any cell $C$ of $\NN^{2}$ not belonging to $\Pc$, 
there exists a path $C=C_1,C_2,\ldots,C_m=D$ with each $C_i\not \in \Pc$ 
such that $D$ is not a cell of $\Pc_{[a,b]}$.  Roughly speaking, 
a simple polyomino is a polyomino with no ``hole'' 
(see \cite[Figure $3$]{Q}).

\medskip

{\bf (0.6)}
Let $\Pc$ be a finite collection of cells of $\NN^{2}$ with $V(\Pc)$ 
its vertex set.  
Let $S$ denote the polynomial ring over a field $K$ whose variables 
are those $x_{a}$ with $a \in V(\Pc)$.
We say that 
an interval $[a,b]$ of $\NN^{2}$
is {\em an interval of $\Pc$} if $\Pc_{[a,b]} \subset \Pc$.
For each interval $[a, b]$ of $\Pc$, 
we introduce the binomial
\[
f_{a,b} = x_{a}x_{b} - x_{c}x_{d}, 
\]
where $c$ and $d$ are the anti-diagonals of $[a,b]$.  
Such a binomial $f_{a,b}$ is said to be an {\em inner $2$-minor} of $\Pc$.
Write $I_{\Pc}$ for the ideal generated by all inner $2$-minors of $\Pc$.
Especially, when $\Pc$ is a polyomino, we say that $I_{\Pc}$ is the 
{\em polyomino ideal} of $\Pc$.

\medskip

Now, one of the most exciting algebraic problems on polyominoes is when
a polyomino ideal is a prime ideal.
It is known (\cite{HM} and \cite{QSS}) that if a polyomino $\Pc$
is simple, then its polyomino ideal $I_{\Pc}$ is a prime ideal.
The polyomino ideals arising from simple polyominoes, however,
turn out to be well-known ideals \cite{OH} arising from
Koszul bipartite graphs.
Thus, form a view point of finding a new class of binomial prime ideals,
it is reasonable to study polyomino ideals of nonsimple polyominoes.
In the present paper, a class of nonsimple polyominoes $\Pc$ 
for which the polyomino ideal $I_{\Pc}$ is a prime ideal
(Theorem \ref{Boston}) and for which $I_{\Pc}$ cannot come from a finite simple graph 
(Theorem \ref{Sydney}) will be presented.

Finally the fact \cite{ES} that a binomial ideal is a prime ideal
if and only if it is a toric ideal (\cite[Chapter 5]{H}) 
explains the reason
why we are interested in polyomino ideals which are prime.

\section{Gr\"obner bases of polyomino ideals}
Let $\Pc$ be a finite collection of cells of $\NN^{2}$.
Let, as before, $S$ denote the polynomial ring 
over a field $K$ whose variables 
are those $x_{a}$ with $a \in V(\Pc)$.
We work with the lexicographical order on $S$ induced by the ordering 
of the variables $x_a$, $a \in V(\Pc)$, such that 
$x_a > x_b$ with $a=(i,j)$ and $b=(k,\ell)$, 
if $i>k$, or, $i=k$ and $j>\ell$.  

We refer the reader to \cite[Chapter $2$]{HH} and \cite[Chapter $1$]{H} 
for basic terminologies and results on Gr\"obner bases.

\begin{Lemma}[\cite{Q}]
\label{quadratic}
Let $\Pc$ be a collection of cells of $\NN^{2}$. 
Then the set of inner $2$-minors of $\Pc$ forms a reduced Gr\"obner basis 
of $I_{\Pc}$ with respect to $<_{\lex}$ if and only if,
for any two intervals $[a,b]$ and $[b,c]$ of $\Pc$, either $[e,c]$ or $[d,c]$ is an interval of $\Pc$, where $d$ and $e$ are the anti-diagonal corners of $[a,b]$.
\end{Lemma}

\begin{Corollary} \label{GB}
Let $\I \subset \NN^2$ be an interval of $\NN^{2}$
and $\Pc$ a convex polyomino which is a subpolyomino of $\Pc_{\I}$. 
Let $\Pc^c = \Pc_{\I} \setminus \Pc$. 
Then the set of inner $2$-minors of $\Pc^c$ forms a reduced Gr\"obner basis 
of $I_{\Pc^c}$ with respect to $<_{\lex}$.
\end{Corollary}

\begin{proof}
Suppose that there exist intervals $[a,b]$ and $[b,c]$ of $\Pc^c$ 
such that neither $[e,c]$ nor $[d,c]$ is an interval of $\Pc^{c}$, 
where $d$ and $e$ are the anti-diagonal corners of $[a,b]$.
Then one can choose a cell $C$ of $\Pc_{[e,c]}$ and a cell $D$ 
of $\Pc_{[d,c]}$ such that $C$ and $D$ 
belong to $\Pc$.
%
%
Now, since $\Pc$ is a polyomino, it follows that 
there is a path of cells 
$C=C_1, C_2, \ldots, C_n=D$
of $\Pc$ connecting $C$ 
with $D$.
Then one of the situations drawn in 
Figure $1$ occurs.
Let $C = [a', a' + (1,1)]$.

\bigskip

\begin{figure}[htbp]
 \begin{minipage}{0.4\hsize}
  \begin{center}
   \includegraphics[width=55mm]{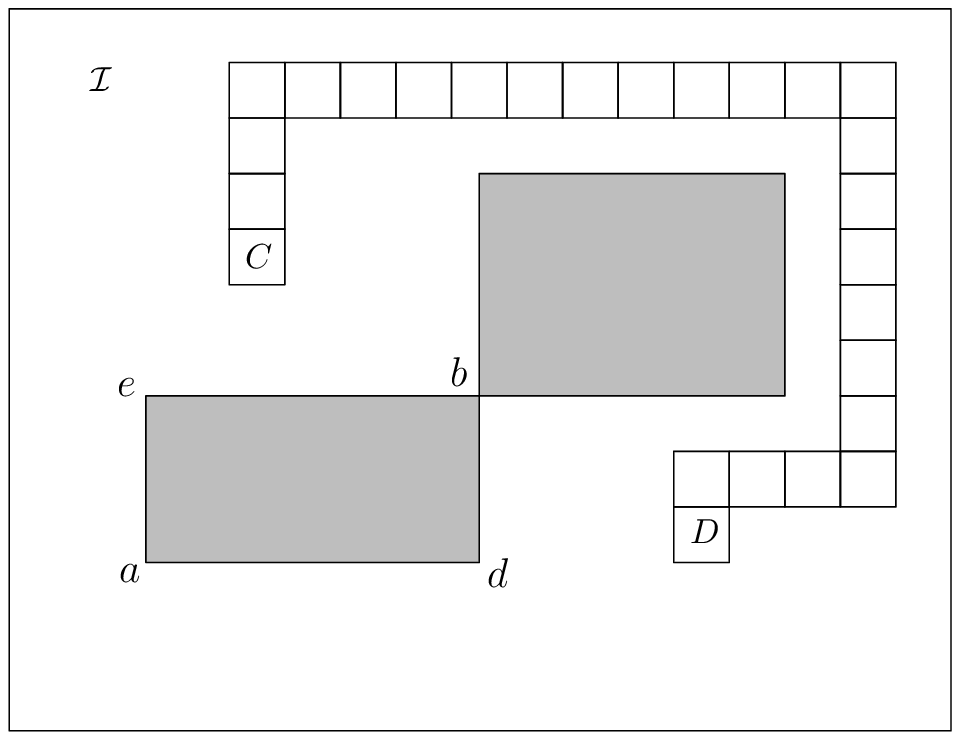}
(1)  
  \end{center}
  \label{GBfig1}
 \end{minipage}
 \begin{minipage}{0.4\hsize}
  \begin{center}
   \includegraphics[width=55mm]{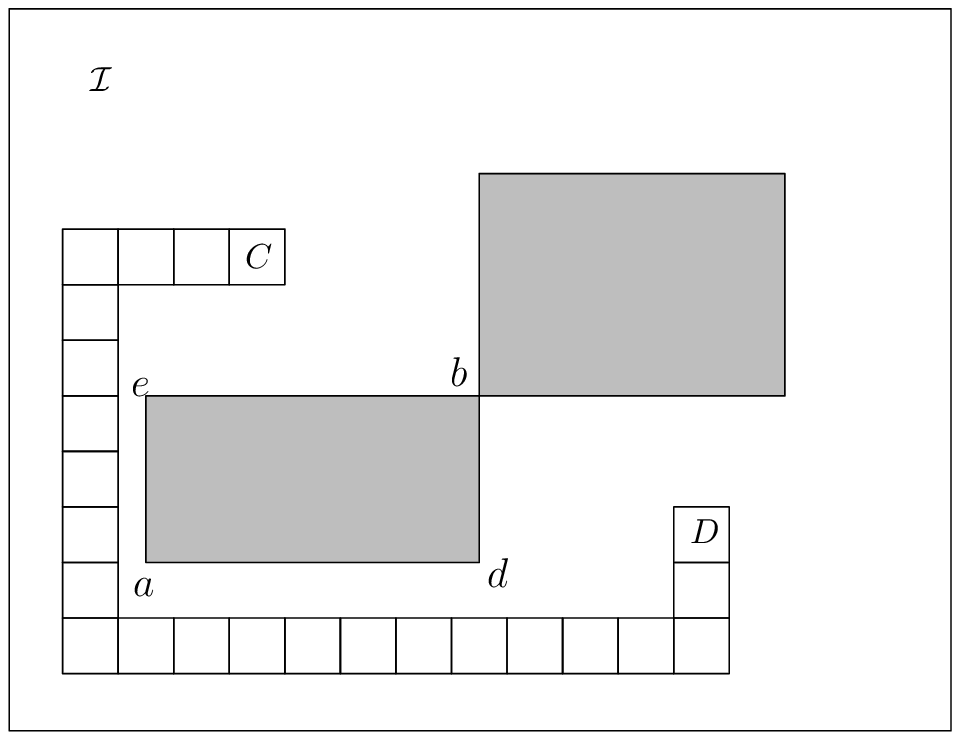}
(2)   
  \end{center}
  \label{GBfig2}
 \end{minipage}
\caption{~}
\end{figure}

\bigskip

\noindent
In other words, there is $1 < j < n$ for which 
$C_{j} = [c', c' + (1,1)]$ satisfies one of the followings:
\begin{enumerate}
\item[(i)]
if $a' = (\xi',\nu'), c' = (\xi'',\nu''), c = (\xi,\nu)$, then
$\nu' = \nu''$ and $\xi'' > \xi$;
\item[(ii)]
if $a' = (\xi',\nu'), c' = (\xi'',\nu''), a = (\xi_{0},\nu_{0})$, then
$\xi' = \xi''$ and $\nu'' < \nu_{0}$.
\end{enumerate}
Since $\Pc$ is convex, it follows that, in (i) one has
$[C,C_{j}] \subset \Pc$, and in (ii) one has
$[C_{j},C] \subset \Pc$.
However, $[C,C_{j}] \cap \Pc_{[a,b]} \neq \emptyset$ in (i)
and $[C_{j},C] \cap \Pc_{[a,b]} \neq \emptyset$ in (ii),
each of which contradicts $\Pc \cap \Pc^{c} = \emptyset$.
\end{proof}

\section{Nonsimple polyominoes whose polyomino ideals are prime}
We now come to the main result of the present paper.

\begin{Theorem}
\label{Boston}
Let $\I \subset \NN^2$ be an interval of $\NN^{2}$
and $\Pc$ a convex polyomino which is a subpolyomino of $\Pc_{\I}$. 
Let $\Pc^c = \Pc_{\I} \setminus \Pc$ and suppose that
$\Pc^c$ is a polyomino. 
Then the polyomino ideal $I_{\Pc^c}$ is a prime ideal.
\end{Theorem}

\begin{proof}
We may assume that $B(\Pc) \cap B(\Pc_{\I}) = \emptyset$;
otherwise, $\Pc$ is a simple polyomino (see Figure $2$)
and, as was stated, the result follows from \cite{HM} and \cite{QSS}.

\bigskip
\bigskip

\begin{figure}[htbp]
 \includegraphics[width = 4cm]{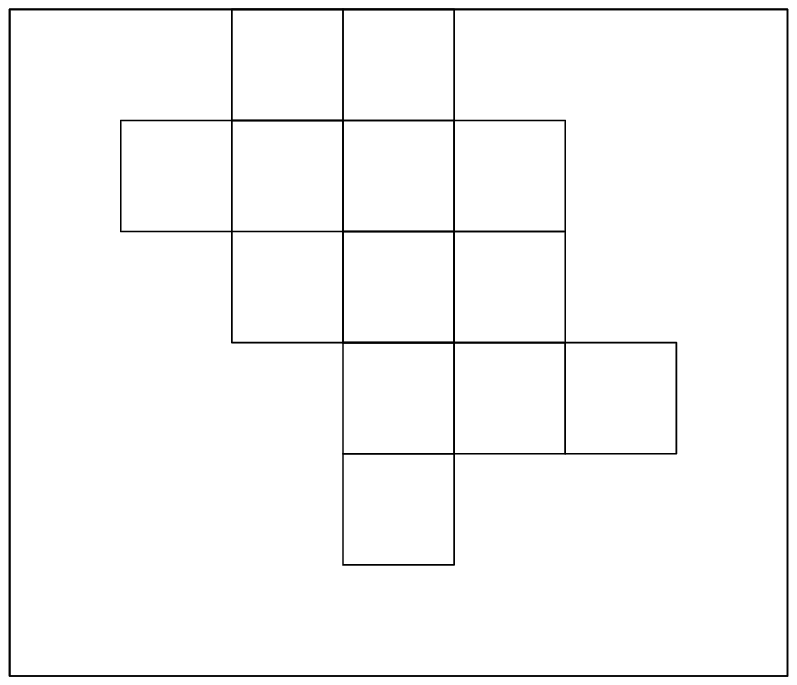}
\caption{~}\label{case1}
\end{figure}

\bigskip

Let $\I=[a,b]$ and $c$ and $d$ be the anti-diagonal corners of $[a,b]$,
where $b$ and $c$ are in horizontal position.  
It follows from Theorem~\ref{GB}
that $x_c$ cannot divide the initial monomial of any binomial
belonging to the reduced Gr\"obner basis of $I_{\Pc}$ 
with respect to $<_{\lex}$.
Hence $x_c$ is a nonzero divisor of 
$S/{\rm in}_{<_{\rm lex}}(I_{\Pc^{c}})$
and thus $x_c$ is a nonzero divisor of 
$S/I_{\Pc^{c}}$ as well. 
Hence the localization map $S/ I_{\Pc^c} \rightarrow (S/ I_{\Pc^c})_{x_c}$
is injective.  Here $(S/ I_{\Pc^c})_{x_c}$ is the localization 
of $(S/ I_{\Pc^c})_{x_c}$ at $x_c$. 
Thus, in order to prove that $S/ I_{\Pc^c}$ is an integral domain, 
it suffices to show that 
$(S/ I_{\Pc^c})_{x_c} = S_{x_c}/ ({I_{\Pc^c}})_{x_c}$ is an integral domain. 
For this, we will 
show that $({I_{\Pc^c}})_{x_c} = I_{\Pc'}$, 
where $\Pc'$ is a simple subpolyomino of $\Pc^c$, 
which guarantees that $({I_{\Pc^c}})_{x_c}$ is 
a prime ideal (\cite{HM} and \cite{QSS}).

Let $\Ac =\{p_1, \ldots, p_n\}$ denote the set of those
$p_i \in V(\Pc^c)$ for which there is an interval $[r_i, q_i]$ of $\Pc^c$ 
whose anti-diagonal corners are $c$ and $p_i$. See Figure~\ref{proof1}.

\bigskip

\begin{figure}[htbp]
\includegraphics[width = 6cm]{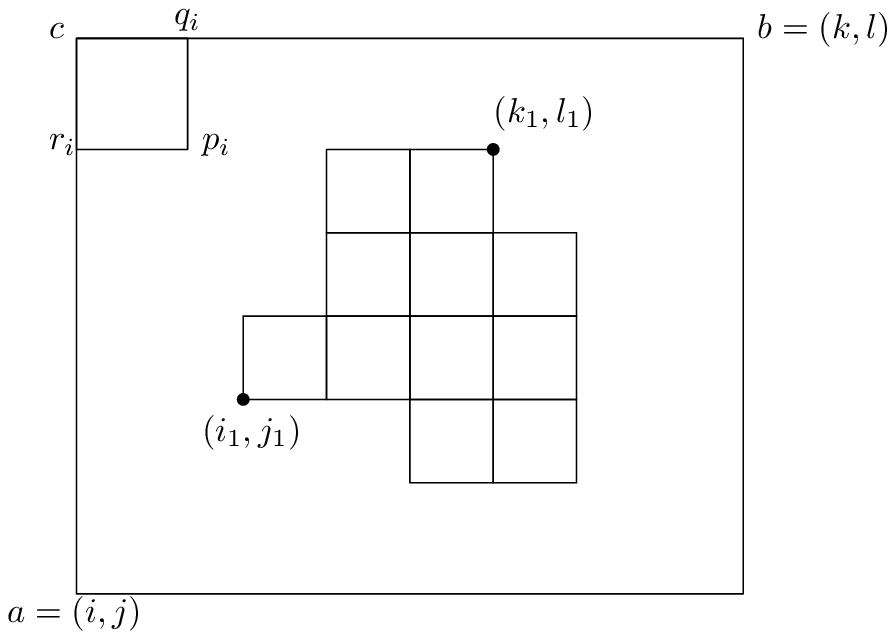}
\caption{~}\label{proof1}
\end{figure}

\bigskip

One has $r_i \in [a,c]$ and $q_i \in [c,b]$. 
Since $x_{r_i} x_{q_i} - x_c x_{p_i} \in I_{\Pc^c}$
and since the variable $x_c$ is invertible in $S_{x_c}$, 
one has 
$x_{p_i} =  x_{q_i} x_{r_i} x_c^{-1}$
in $S_{x_c} / (I_{\Pc^c})_{x_c}$
Thus, in $S_{x_c}/ ({I_{\Pc^c}})_{x_c}$, 
the variables $x_{p_i}$ with $p_i \in \Ac$ can be ignored. 

Let $p_i$ and $p_j$ belong to $\Ac$ 
for which $[p_i,p_j]$ is an interval in $\Pc^c$.  
It then follows that 
the anti-diagonals of $[p_i,p_j]$ are also contained in $\Ac$. 
Thus $f_{p_i,p_j} = x_{p_k} x_{p_\ell} - x_{p_j}x_{p_i}$,
where $p_k$ and $p_\ell$ are the anti-diagonal corners of $[p_i,p_j]$. 

Let $[v, p_i]$ be an interval of $\Pc^c$ with $p_i \in A$ 
and $v \notin A$, then by using the fact that 
$[r_i , p_i] \setminus \{r_i\} \subset A$, it follows that 
the anti-diagonal corner $p_{i'}$ of $[v, p_i]$
which is in horizontal position with $p_i$ belongs to $A$. 
Let $v'$ be the other anti-diagonal corner of $[v, p_i]$. 
Since $r_i = r_{i'}$, the inner $2$-minor  
$x_v x_{p_i} - x_{v'} x_{p_{i'}} \in I_{\Pc^c}$ 
can be written as $x_{r_i} (x_v x_{q_i} - x_{v'} x_{q_{i'}} )$ in  
$(I_{\Pc^c})_{x_c}$. 
Hence $x_v x_{p_i} - x_{v'} x_{p_{i'}}$ 
is a multiple of $x_v x_{q_i} - x_{v'} x_{q_{i'}}$ in $(I_{\Pc^c})_{x_c}$.
Similarly, if $[p_{i}, v]$ is an interval of $\Pc^c$ with $p_i \in A$ 
and $v \notin A$ and if $p_{i'} \in A$ and $v' \not\in A$ are
the anti-diagonal corner of $[p_i, v]$, then
$x_v x_{p_i} - x_{v'} x_{p_{i'}}$ 
is a multiple of $x_v x_{r_i} - x_{v'} x_{r_{i'}}$ 
in $(I_{\Pc^c})_{x_c}$.

Let $\Pc'$ be the collection of cells contained in $\Pc^c$ obtained by removing all the cells that appear in $\bigcup_{i=1}^n \Pc_{[r_i, q_i]}$. 
Let $a=(i,j)$ and $b=(k,\ell)$. Then $c=(i,\ell)$. 
We choose $(i_1, j_1) \in V(\Pc)$ such that, 
for any $(i_2, j_2) \in V(\Pc)$, one has either $i_1< i_2$ or ($i_1=i_2$ 
and $j_1 < j_2$). 
Similarly, we choose $(k_1, \ell_1) \in V(\Pc)$ such that, 
for any $(k_2,\ell_2) \in V(\Pc)$, one has either 
$\ell_1 > \ell_2$ or ($\ell_1 = \ell_2$ and $k_1 > k_2$). 
In $V(\Pc')$, we identify the vertical interval $[a, (i,j_1)]$ 
with $[(i_1,j), (i_{1},j_1)]$, and the horizontal interval
$[(k_1, \ell_1), (k,\ell_1)]$ with $[(k_1,\ell), b]$. 
Then, with this identification and by using the above discussion, 
one has $I_{\Pc'} = (I_{\Pc^c})_{x_c}$. 

Now, what we must prove is that $\Pc'$ is a simple polyomino.
First we claim that $\Pc'$ is a polyomino.  
Let $\mathcal{B}$ be the collection of border cells of $\Pc_{[a,b]}$ 
belonging to $\Pc'$.  Then $\mathcal{B}$ is connected.
Since every cell of $\Pc'$ is connected to at least one of the cells 
belonging to $\mathcal{B}$. Hence $\Pc'$ is connected.  Thus $\Pc'$
is a polyomino, as desired. 
Second, we claim that $\Pc'$ is simple.
Let $\mathcal{J}$ be an interval such that 
$\Pc' \subset \Pc_{[a,b]} \subset \Pc_{\mathcal{J}}$. 
If $\Pc'$ is not a simple polyomino, then one has a cell $D \notin \Pc'$ 
for which every path connecting $D$ 
with a cell not belonging to $\Pc_{\mathcal{J}}$ is interrupted 
by some cell of $\Pc'$. 
The inclusion $\Pc' \subset \Pc^c$ shows that 
$D$ must be a cell of the convex polyomino $\Pc$. 
Then all the cells of $\Pc^c$ whose edge sets intersect 
$B(\Pc)$ must be contained in $\Pc'$, 
which cannot be possible by our construction of $\Pc'$. 
Hence $\Pc'$ is simple, as required.
\end{proof}

\section{toric ideals of finite graphs}
As was stated in Introduction, 
one of the most exciting algebraic problems on polyominoes is when
a polyomino ideal is a prime ideal.
The fact (\cite{HM} and \cite{QSS}) that the polyomino ideals of 
simple polyominoes are prime seems to be of interest.  
However, it turns out that these binomial ideals belong to a subclass of 
binomial ideals arising from Koszul bipartite graphs (\cite{OH}).  
Thus, form a view point of finding a new class of binomial prime ideals,
the study of polyomino ideals of nonsimple polyominoes
is indispensable.  

In fact, the polyomino ideals of
Theorem \ref{Boston} {\em cannot} come from finite simple graphs.
(We say that a binomial ideal $I$ comes from a finite simple graph if
$I$ coincides with a toric ideal \cite{OHbinomial} arising from a finite simple graph.)  
More generally, we can show that

\begin{Theorem}
\label{Sydney}
Let $\I \subset \NN^2$ be an interval of $\NN^{2}$
and $\Pc$ a simple polyomino which is a subpolyomino of $\Pc_{\I}$. 
Let $\Pc^c = \Pc_{\I} \setminus \Pc$ and suppose that
$\Pc^c$ is a polyomino. Then its polyomino ideal
cannot come from a finite simple graph. 
\end{Theorem}

\begin{proof}
Let $\mathcal{J}$ be the smallest interval in $\NN^2$ such that $\Pc \subset \J$. 
We choose $x_1, \ldots, x_{16}$ belonging to $V(\Pc^c)$, as shown in Figure~\ref{1}, where $\Pc$ is shown by grey region and where $\mathcal{J} = [x_{10}, x_{7}]$.

\begin{figure}[htbp]
\includegraphics[width = 8cm]{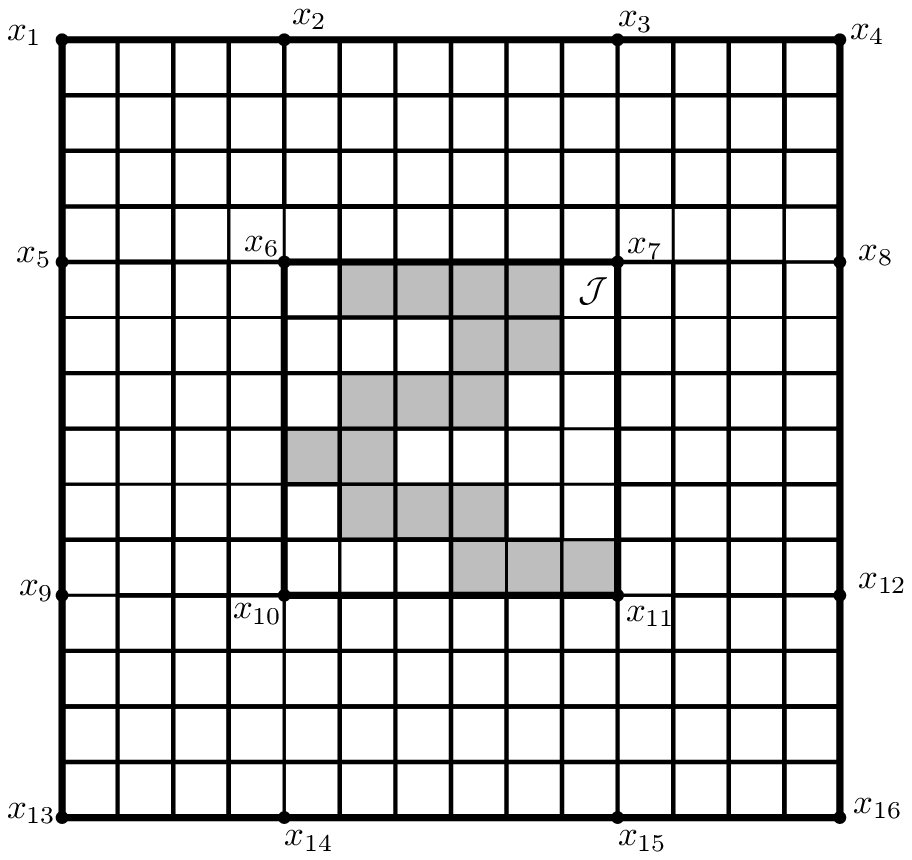}
\caption{Polyomino $\Pc^c$}\label{1}
\end{figure}

Assume that there exists a finite simple graph $G$ with vertex set $V(G)$ and edge set $E(G)$ such that the toric ideal $I_G$ arising from $G$ is equal to $I_{\Pc}$. 
Let $K[G] = K[t_i t_j | \{i,j\} \in E(G)]$ be the edge ring of $G$. Then there exists an isomorphism $\phi: K[\Pc] \rightarrow   K[G]$ such that for each $ x_a \in K[\Pc]$ there exists a unique edge $\{i,j\} \in E(G)$ with $\phi(x_a) = t_i t_j$.

The 2-minor $x_2 x_7 - x_3 x_6$ is an inner minor of $\Pc^c$ and hence $\phi (x_2 x_7) = \phi(x_3 x_6)$. Let $\phi(x_2) = t_i t_j$. Then $\phi (x_7)= t_k t_l $ where $i,j,k,l$ are pairwaise distinct vertices of $G$ and $\{i,j\}, \{k,l\} \in E(G)$.  Then $\phi(x_3 x_6) = t_i t_j t_k t_l$ which shows that we have one of the following possibilities:

\begin{enumerate}
\item[(i)] $\phi(x_3) = t_i t_k$ and $\phi(x_6) = t_j t_l$;
\item[(ii)] $\phi(x_3) = t_i t_l$ and $\phi(x_6) = t_j t_k$;
\item[(iii)] $\phi(x_3) = t_j t_k$ and $\phi(x_6) = t_i t_l$;
\item[(iv)] $\phi(x_3) = t_j t_l$ and $\phi(x_6) = t_i t_k$.
\end{enumerate}

We may assume that $\phi(x_3) = t_i t_k$ and $\phi(x_6) = t_j t_l$. The discussion for other cases is similar. By using the inclusion $x_1 x_6 - x_2 x_5 \in I_{\Pc^c}$ and that $\phi(x_2) = t_i t_j$ and $\phi(x_6) = t_j t_l$, we see that $\phi(x_1) = t_i t_p$ and $\phi (x_5) = t_l t_p$where $\{i,p\} , \{l,p\} \in E(G)$ for some $p \in V(G) \setminus \{i,j,k,l\}$. Note that $p \neq k$ because otherwise $\phi(x_5)= \phi(x_7) = t_k t_l$, which is not possible. Now  from $x_5 x_{10} - x_6 x_9  \in I_{\Pc^c} $ and $\phi(x_5) = t_p t_l$, $\phi(x_6) = t_j t_l$, we obtain $\phi(x_{10}) = t_j t_q$ and $\phi(x_9) = t_p t_q$ for some $q \in V(\Pc^c) \setminus \{i,p,l,j\}$. Continuing in the same way, from $x_9 x_{14} - x_{10}x_{13} \in I_{\Pc^c}$
and $\phi(x_9) = t_p t_q$ and  $\phi(x_{10}) = t_j t_q$, we get $\phi(x_{14}) = t_r t_j$ and $\phi(x_{13}) = t_r t_p$ for some $r \in V(\Pc^c) \setminus \{i,j,l,p,q\}$. 
Then, by using $x_{10} x_{15} - x_{11}x_{14} \in I_{\Pc^c}$, $\phi(x_{10}) = t_j t_q$ and $\phi(x_{14}) = t_r t_j$, we get  $\phi(x_{15}) = t_s t_r$ and  $\phi(x_{11}) = t_s t_q$ for some $s \in V(\Pc^c) \setminus \{j,p,q,r\}$. 

Furthermore, by using $x_3 x_8 - x_4 x_7 \in I_{\Pc^c}$,
$\phi(x_{3}) = t_i t_k$ and $\phi(x_{7}) = t_k t_l$,
we obtain  $\phi(x_4) = t_i t_y$ and $\phi (x_8) = t_l t_y$ for some $y \in V(G) \setminus \{i,k,l,j,p\}$. 
Similarly, from $x_{7} x_{12} - x_{11}x_{8} \in I_{\Pc^c}$, $\phi(x_{7}) = t_k t_l$, $\phi(x_{8}) = t_y t_l$ and $\phi(x_{11}) = t_s t_q$, it follows that $t_k | t_s t_q$. 
Thus one has either $k=s$ and $\phi(x_{12}) = t_q t_y $ or $k=q$ and $\phi(x_{12}) = t_s t_y $.

Let $k=s$. Then $\phi(x_6 x_{11}- x_7 x_{10})= (t_j t_l) (t_k t_q) - (t_k t_l) (t_j t_q) =0$, 
which guarantees $x_6 x_{11}- x_7 x_{10} \in I_G$.
However, one has $x_6 x_{11}- x_7 x_{10} \notin I_{\Pc^c}$, because it is not an inner minor of $\Pc^c$, and it gives us a contradiction to our assumption $I_G = I_{\Pc^c}$. 
Hence $k=q$ and $\phi(x_{12})= t_s t_y$. But then $x_{11} x_{16}- x_{12} x_{15} \in I_{\Pc^c} = I_G$, $\phi(x_{12} x_{15}) = (t_s t_y )(t_s t_r)$ and $\phi(x_{11})= t_s t_k$. 
Thus one has either $k=r$ or $k=s$, which is not possible;
otherwise either $\phi(x_{11})= t_s t_r = \phi(x_{15})$ or $\phi(x_{11})= t_s ^2$. 
As a result, we conclude that $I_G \neq I_{\Pc^c}$ for any finite simple graph $G$.
\end{proof}

Finally, it may be conjectured that the polyomino ideal $I_{\Pc}$ of a polyomino $\Pc$
comes from a finite simple graph if and only if $\Pc$ is nonsimple.
Furthermore, Theorem \ref{Boston} might be true when $\Pc$ is simple.

\end{document}